\documentclass[12pt]{article}

\usepackage{amsmath, amsthm, color}
\usepackage{amssymb} 
\usepackage{amscd}
\usepackage{enumerate}

\newtheorem{thm}{Theorem}[section]

\newtheorem{lemma}{Lemma}[section]
\newtheorem{defn}{Definition}[section]

\title{A note on locally metric connections}
\author{  Mihail Cocos
	}

\date{\today}

\begin{document}
	\maketitle

\begin{abstract}
In this paper we give necessary and sufficient conditions for a connection in a plane bundle above a surface to be locally metric. These conditions are easy to be verified using any local frame. Also, as a global result, we give a  necessary condition for two connections to be metric equivalent in terms of their Euler class.
\end{abstract}

\section{Introduction: Locally metric connections}

Throughout this paper $E$ will denote a real vector bundle of rank $m$ over a manifold $M$  of dimension $n.$

\begin{defn}
	A connection $D$ in $E$ is called  locally metric if and only if there is a bundle metric that is parallel with respect to $D$ in the neighborhood of any point $p \in M.$
\end{defn}
\noindent It is not hard to prove that a connection is locally metric if its restricted holonomy group $Hol_0(D)$ is pre-compact(see \cite{BM}). However, since the explicit calculation of holonomy is very difficult, we will not make use of this variant of the definition. The literature for locally metric connections is scarce. There are not very many practical criteria to verify whether a connection is locally metric. Some promising results can be found in (\cite{RA1,RA2,VZ1,VZ2,T,NC,K}).
From a pure mathematical point of view, the investigation of the set of symmetric and locally metric connections in the tangent bundle of a manifold is related to a long standing conjecture of Chern for affinely flat manifolds. Chern's statement conjectures that the Euler characteristic of a compact affinely flat manifold is zero. The author showed in  (\cite{MC}) that for a locally metric connection there is a natural cohomology class of $M$ that coincides with the Euler class of the bundle $E$ in the case when the connection is globally metric. He also proved that if the set of locally metric connections in the tangent bundle of an affinely flat manifold is path connected, then Chern's statement is true(see \cite{MC}).\\
If two symmetric and locally metric connections in the tangent bundle of a manifold $M$ share a parallel metric in the neighborhood of any point then, by the fundamental theorem of Riemannian geometry, they are equal. If the symmetry is dropped one can no longer conclude that they are equal and the relationship becomes a weaker equivalence relation. The global result of this paper gives a necessary condition for two connections to be equivalent in terms of their Euler class.

\bigskip

We will now establish some conditions for a connection to be metric in terms of local frames.
\begin{lemma} \label{skew1}
Let $D$ be connection in $E.$ Then $D$ is locally metric if and only if in the neighborhood of any point there exist a local frame of the bundle $$\sigma=(\sigma_1,\sigma_2,...,\sigma_m)$$
such that the connection matrix $ \theta$  with respect to $\sigma$ is skew symmetric.

\end{lemma}
\begin{proof}

The only if part is obvious. For the if part let $$\sigma=(\sigma_1,\sigma_2,...,\sigma_m)$$ as in hypothesis. Take the metric $g$ that makes the frame $\sigma$ orthonormal, that is

$$g(\sigma_i,\sigma_j)=\delta_{ij}.$$ Differentiating $g$ in the direction $X \in TM,$ we get $$ (D_Xg)(\sigma_i,\sigma_j)=0-\left(\theta_{ij}(X)+\theta_{ji}(X)\right) = 0,$$ hence $g$ is parallel with respect to $D.$

\end{proof}

For the sake of concreteness let us give a couple of examples of connections which are relevant to the theoretical results of this paper. Let $M=\mathbb{T}^2$ be the two dimensional torus and let $f=(f_1,f_2)$ a global frame of commuting vector fields in $TM.$ The dual frame will be denoted $f^*=(f^1,f^2).$
Consider the connection $D$ that has  its connection matrix with respect to $f$  

$$ \theta=\begin{Vmatrix}
f^1 & 0\\ 0 & 0
\end{Vmatrix}$$
This is a flat, symmetric connection in $TM$ hence locally metric. However there is no global metric on $M$ that will have $D$ as its Levi Civita connection. Assume, by contradiction, that there is a global metric $g$ that is preserved by $D$ and consider $\gamma=\gamma(t), \ \ -\infty< t<\infty$  an integral curve of $f_1.$ On $M$ we have a globally defined smooth function defined as $$ h=g(f_1,f_1),$$ and if we look at its restriction $h(t)$ to $\gamma$, then $h(t)$ satisfies the differential equation $$h'(t)=2h(t),$$ and hence it is an exponential. Consequently $h$ is not bounded on $M$ which gives a contradiction.\\

\bigskip

 Before giving another example we need to make a definition

\begin{defn}\label{metricequiv}
	Let $D$ and $\tilde{D}$ be two locally metric connections in a vector bundle $E.$ We say they are metric equivalent if and only if in the neighborhood of any point there exist a local bundle metric $g$ such that $$ D g=\tilde{D} g=0.$$
\end{defn}

\noindent Our second example is a connection $\tilde{\nabla}$ in the tangent bundle of a generic Riemannian manifold $(M,g)$ that is metric equivalent ($\nabla g \equiv \tilde{\nabla}g \equiv 0$) to the Levi Civita connection $\nabla$ but not symmetric. 
We will base our example on Theorem 2.1 in (\cite{MMT}). Take $u$ a one form on $M$ and let $u^{\#}$ be its metric dual with respect to $g$. Define \begin{equation}\label{MukutConnection}\tilde{\nabla}_XY=\nabla_XY+u(Y)X-g(X,Y)u^{\#}. \end{equation} It is easy to verify that equation (\ref{MukutConnection}) defines a connection compatible with the metric $g$ and that

$$T_{\tilde{\nabla}}(X,Y)=u(Y)X-u(X)Y, $$ 

\noindent and therefore $\tilde{\nabla}$ is non-symmetric if $u \neq 0.$

\bigskip

Next let us note that if a connection preserves a metric (locally) then it also preserves a volume form, hence we have the following criterion

\begin{lemma}\label{VolCrit}

Let $D$ be a connection in a vector bundle $E$, $\theta $ its connection matrix with respect to a local frame $\sigma$ and $\Omega $ its curvature matrix. Then $D$ preserves a local volume form if and only if  $$ Tr \,\, \Omega=d( Tr\ \theta)=0.$$

\end{lemma}

\begin{proof}
	First let us note that $$Tr \,\, \Omega=d( Tr\ \theta)$$ follows immediately from Cartan's equation  $$\Omega=d\theta+\theta \wedge \theta$$ by taking the trace of both sides.
	
	\noindent Next let  $\omega$ be a local volume form and $e=(e_1,e_2,...,e_m)$ be a positive local frame with respect to $\omega.$ We have \begin{equation}\label{volume} \omega= f e^1\wedge e^2 \cdots \wedge e^m,\end{equation} where $ f>0$ is a local smooth function.
Taking the covariant derivative of $\omega$ in the direction of the vector field $X$ we get

\begin{equation}\label{volume2}  D_X \omega =X(f)e^1\wedge e^2 \cdots \wedge e^m +f\sum_{k=1}^m e^1\wedge e^2 \cdots \wedge D_X e^k \wedge \cdots \wedge e^m .\end{equation}

Since $$ D_X e^k= -\theta_{sk}(X) e^s$$ and by using (\ref{volume2}) we obtain \begin{equation}\label{volume3}
D_X\omega=(X(f)-f\sum_{k=1}^m \theta_{kk}(X)) e^1\wedge e^2 \cdots \wedge e^m \end{equation} which shows that $$ D_X\omega =0$$ is equivalent to \begin{equation}\label{volume4} d(\ln\ f) = Tr \ \theta .\end{equation}

\end{proof}

\noindent For the case of plane bundles we will be able to give a practical test for local metrizability. For this criteria we will need the following linear algebra lemmas

\begin{lemma}\label{nice0}
Let $U$ be a nonsingular $2 \times 2$ matrix. Let $A,B$ two matrices with determinant equal to one that satisfy $$A^{-1}UA=B^{-1}UB=kJ,$$ with $J=\begin{Vmatrix} 0 & 1\\ -1 & 0 \end{Vmatrix}.$ Then $$B=AS,$$ with $S$ orthogonal.
\end{lemma}
\begin{proof} First, let us note the following easy to prove identity for $2 \times 2$ matrices

\begin{equation}\label{matrixident1} XJX^{-1}=XX^TJ,\end{equation} where $X$ is any $2 \times 2$ matrix with determinant equal to one and $J$ is as in the hypothesis. Using the hypothesis we get $$U=kAJA^{-1}=kBJB^{-1},$$ and by using (\ref{matrixident1}) we get \begin{equation}\label{matrixident2} AA^T=BB^T \end{equation} and therefore \begin{equation}\label{matrixident3} A^{-1}B=A^T(B^T)^{-1} .\end{equation}

Let's set $$S=A^{-1}B=A^T(B^T)^{-1}.$$ We have $$ SS^T=A^T(B^T)^{-1}B^T(A^T)^{-1}=I,$$
which proves the lemma.

\end{proof}

\begin{lemma}\label{Skew}

Let $U$ be a $2 \times 2$ a nonzero real matrix. Then there exists a matrix $A$ such that $A^{-1}UA$ is skew if and only if $U$ has purely imaginary, non zero eigenvalues. The matrix $A$ can be chosen such that its determinant is one and its entries are smooth in terms of the entries of $U.$

\end{lemma}
\begin{proof}

For the only if part let's assume that there exist a matrix $A$ such that $$A^{-1}UA=\begin{Vmatrix} 0 & a \\ -a & 0 \end{Vmatrix}$$ with $a > 0.$ Then the characteristic equation of $U$ is $$ \lambda^2 +a^2=0$$ and therefore its eigenvalues are $\pm ia.$

For the if part of the lemma let us first note that we can always find a positive definite, symmetric matrix $S$ such that \begin{equation}\label{Skew1} US+SU^T=0.\end{equation}
Let $A$ be the only symmetric, positive definite matrix satisfying  \begin{equation}\label{Skew2} AA^T=S.\end{equation}
We have $$A(A^{-1}UA+A^TU^T(A^{-1})^T)A^T=US+SU^T=0,$$ and therefore $$A^{-1}UA+(A^{-1}UA)^T=0.$$
\end{proof}
\begin{thm}\label{nice}
	Let $E$ be a plane bundle over a surface $\Sigma$ and $D$ a connection in $E.$ Assume that the curvature of $D$ at $p \in \Sigma$ is nonzero. Let $e=(e_1,e_2)$ be a local frame in $E$ around $p \in \Sigma $, $\theta$ its connection matrix and $\Omega$ its curvature matrix with respect to $e.$ Let $f=(f^1,f^2)$ be a local frame of one forms around $p \in \Sigma$ and  $U$ be the matrix with real entries defined by the equation $$ \Omega= (f^1\wedge f^2) U.$$
Then the $D$ is locally metric if and only if the following two conditions are satisfied
	
	\begin{enumerate}[(a)]
	
	\item $U$ has purely imaginary eigenvalues.
	\item If $A$ is a matrix as in Lemma (\ref{Skew}) then the connection matrix is skew symmetric with respect to the frame $f=eA.$

	\end{enumerate}
	
\end{thm}
\begin{proof}
For the "only if part" since $D$ is assumed locally metric, according to Lemma \ref{skew1}, there exist a frame $\tau$ such that the connection matrix $ \psi$ with respect to this frame is skew and consequently the curvature matrix $\Psi$. Let $B$ be the matrix defined by $$ \tau=eB.$$ We have $$ \Psi=(f^1\wedge f^2) B^{-1}UB$$ and hence $B^{-1}UB$ is skew and non zero. According to Lemma \ref{Skew} it follows that $U$ has purely imaginary eigenvalues.
Now let $A$ be a matrix as in Lemma \ref{Skew}. Both $A$ and $B$ can be chosen such that their determinant at every point is one. Since $B^{-1}UB$ and $A^{-1}UA$ are both skew, according to Lemma \ref{nice0} we have
 $$ B=AS $$
with $S$ orthonormal. Let $\theta$ be the connection matrix with respect to $f.$ It follows that

$$\theta=S^{-1}dS+S^{-1}\psi S.$$ But since $S$ is orthonormal, from $S^TS=I$(by differentiating) it follows that $S^{-1}dS$ is skew as well as $S^{-1}\theta S$ and therefore the connection matrix $\theta$ is skew.
The "if" part follows from Lemma \ref{skew1}.
\end{proof}

For a plane bundle over a surface $\Sigma$ endowed with a connection $D$, Theorem \ref{nice} and Lemma \ref{Skew}, give an algorithm that allows us to determine whether a connection is locally metric. Here is the description of the algorithm

\begin{enumerate}[(a)]
\item Take a local frame $e=(e_1, e_2)$ and calculate the curvature matrix $\Omega$ with respect to this frame.
\item Take any "volume" form $\omega$  and factor it out of $\Omega$ $$\Omega= \omega \,  U.$$
\item Calculate the eigenvalues of $U$. If at a point of the chosen neighborhood the eigenvalues are not purely imaginary, then the connection is not metric.
\item If the eigenvalues are purely imaginary, find a symmetric positive solution $S$ to the system $$US+SU^T=0,$$
and calculate its square root $\sqrt{S}=A.$
\item Take the frame $f$ defined by $$f=eA$$ and test whether the connection forms with respect to $f$ are skew. If they are not then the connection is not locally metric.
\item If the connection forms with respect to $f$ are skew, then the metric that makes $f$ orthonormal is compatible with the connection.
\end{enumerate}

\section{On the metric equivalence of connections}
In this section we give a necessary condition for two connections to be metric equivalent ( See Definition \ref{metricequiv}).\\

\begin{thm}

If $D$ and $\tilde{D}$ are two metric equivalent locally metric connections in $E$ then their Euler class is the same.

\end{thm}

\begin{proof}
Let $\pi:M \times \mathbb{R} \rightarrow M,$ denote the projection $$ \pi(p,t)=p$$ and let $\tau=\pi^*(E)$ denote the pullback of the bundle of $E.$ Let $D^*=\pi^*(D)$ and $\tilde{D}^*=\pi^*(\tilde{D})$ be the pullback of the two connections from $E$ to $\tau.$ Consider the linear combination \begin{equation}\label{connection1} D_t^*=(1-t)D^*+t \tilde{D}^* \end{equation} and as usual define a connection $\mathbb{D}$ in $\tau$ by \begin{equation}\label{connection2} \mathbb{D} \sigma (p,t)= D_t^* \sigma \end{equation}
First we need to show that the connection $\mathbb{D}$ is locally metric. Let $(p,t) \in M\times \mathbb{R}$ be an arbitrary point. Since $D$ and $\tilde{D}$ are metric equivalent we can find a bundle metric $g$ defined in a neighborhood $p \in U \subset M $ such that \begin{equation}\label{connection3} Dg=\tilde{D}g=0. \end{equation}
Let $\sigma=(\sigma_1,\sigma_2,..., \sigma_m)$ be an orthonormal frame with respect to $g.$ If we denote by  $\theta$ and $\psi$  the connection matrix of $D$ and $\tilde{D}$ with respect to $\sigma,$ then clearly they are both skew-symmetric. With respect to the pullback frame $\pi^*(\sigma)$ the connection $\mathbb{D}$ has the connection matrix $\omega$ and satisfies the equation \begin{equation}\label{connection4} \omega=(1-t)\theta+t\psi,\end{equation} and therefore is skew symmetric. By Lemma (\ref{skew1}) it follows that $\mathbb{D}$ is locally metric. According to Lemma 3.1 in (\cite{MC}) its Euler form is well defined and closed. Let us denote the Euler class of $\mathbb{D}$ by $\mathcal{A}$ and the the Euler class of $D$ and $\tilde{D}$ by $\mathcal{A}_k$, for $k=0,1.$

\noindent We define a family of maps \[ i_t:M \rightarrow M \times \mathbb{R} \]
\noindent by \[ i_t(p)=(p,t). \]
\noindent We have
\[ i_0^{*}\mathcal{A}=\mathcal{A}_0 \]
\noindent and \[ i_1^* \mathcal{A}=\mathcal{A}_1 .\]
\noindent Because the two maps $i_0$ and $i_1$ are homotopic and $\mathcal{A}$ is closed, they induce the same map in cohomology and it follows that
\[ \mathcal{A}_0-\mathcal{A}_1 \]
is exact on $M$ and the conclusion of the theorem follows.
\end{proof}

\noindent Department of Mathematics,
 Weber State University, Ogden, UT  84408, USA,   
\  \textbf{e-mail:} \ {\tt mihailcocos@weber.edu}


\begin{thebibliography}{10}

\bibitem{RA1} R. Atkins, "When is a connection a metric connection", New Zealand Journal of Mathematics Volume 38 (2008), 225-238


\bibitem{RA2} R. Atkins, Z. Ge, "An inverse problem in the calculus of variations and the characteristic curves of connections on SO(3) bundles", Canad. Math. Bull. Vol. 38(2),1995 pp.129-140

\bibitem{BM} F. Belgun, A. Moroianu,
On the irreducibility of locally metric connections,
Journal f\"ur die reine und angewandte Mathematik (Crelle’s Journal)
, February 2014, DOI:
10.1515/crelle-2013-0128


\bibitem{NC}   K. S. Cheng, W. T. Ni, " Necessary and sufficient conditions for the existence of metrics in two-dimensional affine manifolds" Chinese J. Phys. 16 (1978), 228–232.


\bibitem{MC} M. Cocos, "The deformation of flat connections and affine manifolds", Geom Dedicata (2010)144:71-78 

\bibitem{K} O. Kowalski, M. Belger,  "Metrics with the Prescribed Curvature Tensor and all Its Covariant Derivatives at One Point", Mathematische Nachrichten, Volume 168, Issue 1,  pages 209-225, 1994



\bibitem{T} G.Thompson, "Local and global existence of metrics in two-dimensional affine manifolds", Chinese J. Phys. 19, 6 (1991), 529–532.





\bibitem{MMT} M.M. Tripathi, "A new connection in a Riemannian manifold", Int. Electron. J. Geom.  1  (2008),  no. 1, 15–24


\bibitem{VZ1} Alena Van\v{z}urov\'a, Petra \v{Z}\'a\v{c}kov\'a,
"Metrizability of connections on two-manifolds"
Acta Universitatis Palackianae Olomucensis. Facultas Rerum Naturalium. Mathematica
, Vol. 48 (2009), No. 1, 157--170

\bibitem{VZ2}  Van\v{z}urov\'a, A., \v{Z}\'a\v{c}kov\'a, P.: Metrization of linear connections. Aplimat, J. of Applied Math. (Bratislava) 2, 1 (2009), 151–163.





\end{thebibliography}
\end{document}